\documentclass[a4paper]{amsart}
\usepackage[utf8]{inputenc}

\usepackage{amssymb, amsmath}
\usepackage{amsfonts}
\usepackage{float}
\usepackage{graphics}
\usepackage{subfigure}
\usepackage{color}
\usepackage{graphicx}
\usepackage{enumitem}
\usepackage{pgfplots}

\newcommand{\abs}[1]{\lvert #1 \rvert}
\newcommand{\R}{\text{R}}
\newcommand{\norm}[1]{\| #1 \|}
\newcommand{\ip}[1]{\langle #1 \rangle}
\def\bfv{{\bf v}}

\renewcommand{\Re}{\operatorname{Re}}
\renewcommand{\Im}{\operatorname{Im}}

\newtheorem{lemma}{Lemma}
\newtheorem{theorem}{Theorem}
\newtheorem*{remark}{Remark}
\begin{document}
\title{Stability and transitions of the second grade Poiseuille flow}

\author[Ozer]{Saadet Ozer}
\address[SO]{Department of Mathematics, Istanbul Technical University, 34469, Istanbul, Turkey}
\email{saadet.ozer@itu.edu.tr}

\author[Sengul]{Taylan Sengul}
\address[TS]{Department of Mathematics, Marmara University, 34722 Istanbul, Turkey}
\email{taylan.sengul@marmara.edu.tr}

\begin{abstract}
In this study we consider the stability and transitions for the Poiseuille flow of a second grade fluid which is a model for non-Newtonian fluids.
We restrict our attention to flows in an infinite pipe with circular cross section that are independent of the axial coordinate.

We show that unlike the Newtonian ($\epsilon=0$) case, in the second grade model ($\epsilon \neq 0$ case),
the time independent base flow exhibits transitions as the Reynolds number $\R$ exceeds the critical threshold $\R_c \approx 4.124 \epsilon^{-1/4}$ where $\epsilon$ is a material constant measuring the relative strength of second order viscous effects compared to inertial effects.

At $\R=\R_c$, we find that generically the transition is either continuous or catastrophic and a small amplitude, time periodic flow with 3-fold azimuthal symmetry bifurcates.
The time period of the bifurcated solution tends to infinity as $\R$ tends to $\R_c$.
Our numerical calculations suggest that for low $\epsilon$ values, the system prefers a catastrophic transition where the bifurcation is subcritical.

We also find that there is a Reynolds number $\R_E$ with $\R_E < \R_c$ such that for $\R<\R_E$, the base flow is globally stable and attracts any initial disturbance with at least exponential speed.
We show that $\R_E \approx 12.87$ at $\epsilon=0$ and $\R_E$ approaches $\R_c$ quickly as $\epsilon$ increases.

\end{abstract}
\maketitle

\section{Introduction}
Certain natural materials manifest some fluid characteristics that can not be represented by well-known linear viscous fluid models.
Such fluids are generally called non-Newtonian fluids.
There are several models that have been proposed to predict the non-Newtonian behavior of various type of materials.
One class of fluids which has gained considerable attention in recent years is the fluids of grade $n$ \cite{Massoudi2008199,Hayat20041621,passerini2000,ozer1999,guptastability,wan2008,fetecau2005}.
A great deal of information for these types of fluids can be found in \cite{dunn1995}.
Among these fluids, one special subclass associated with second order truncations is the so called second-grade fluids.
The constitutive equation of a second grade fluid is given
by the following relation for incompressible fluids:
\begin{equation*}
  {\bf t}=-p{\bf I}+\mu {\bf A}_1+\alpha_1{\bf A}_2+\alpha_2{\bf A}_1^2,
\end{equation*}
where ${\bf t}$  is the stress tensor, $p$ is the pressure, $\mu$ is the classical viscosity, $\alpha_1$ and $\alpha_2$ are the material coefficients.
${\bf A}_1$ and ${\bf A}_2$ are the first two Rivlin-Ericksen tensors defined by
\begin{align*}
& {\bf A}_1=\nabla {\bf v}+\nabla {\bf v}^T, \\
& {\bf A}_2=\dot{{\bf A}}_1+{\bf A}_1 \nabla {\bf v}+\nabla {\bf v}^T {\bf A}_1,
\end{align*}
where ${\bf v}$ is the velocity field and the overdot represents the material derivative with respect to time.
This type of constitutive relation was first proposed in \cite{coleman1960}.
The following conditions:
\[
  \alpha_1+\alpha_2=0, \quad \mu \geq 0,\quad \alpha_1\geq0,
\]
must be satisfied for the second-grade fluid to be entirely consistent with classical thermodynamics and the free energy function achieves its minimum in equilibrium \cite{dunn1974}.

Equation of motion for an incompressible second grade Rivlin Ericksen fluid is represented as:
\begin{equation*}
\begin{aligned}
  & \rho(\bfv_t+{\bf w}\times\bfv+\nabla\frac{|\bfv|^2}{2})=-\nabla p+\mu\Delta\bfv+\alpha[\Delta\bfv_t+\Delta{\bf w}\times
  \bfv+\nabla(\bfv\cdot\Delta\bfv+\frac{1}{4}|{\bf A}_1|^2)],\\
  & \nabla \cdot \bfv = 0.
\end{aligned}
\end{equation*}
where $\rho$ is the density, $\alpha=\alpha_1=-\alpha_2$, represents the second order material constant. Subscript $t$ denotes the
partial derivative with respect to time, $\bf w$ is the usual vorticity vector defined by
\[
  {\bf w}=\nabla\times \bfv.
\]
We next define the non-dimensional variables:
\[
  \bfv^*=\frac{\bfv}{U},\quad
  p^*=\frac{p}{\rho}{U^2},\quad
  t^*=\frac{t U}{L},\quad
  {\bf x}^*=\frac{{\bf x}}{L},
\]
where $U$ and $L$ are characteristic velocity and length, respectively. By letting $\epsilon$ represent the second order non-dimensional material constant which measures the relative strength of  second order viscous effects compared to inertial effects and defining the Reynolds number,
\begin{equation*}
\R=\frac{\rho U L}{\mu},\qquad \epsilon=\frac{\alpha}{\rho L^2},
\end{equation*}
the equation of motion, with asterisks omitted, can be
expressed as:
\begin{equation} \label{field1}
\nabla {\bar p}=\frac{1}{\R}\Delta\bfv+\epsilon(\Delta{\bf
w}\times\bfv+\Delta\bfv_t)-\bfv_t-{\bf w}\times\bfv,
\end{equation}
where the characteristic pressure $\bar p$ is defined as:
\[
  \bar p=p+\frac{|\bfv|^2}{2}-\epsilon (\bfv \Delta\bfv+\frac{1}{4}|{\bf A}_1|^2).
\]
Taking curl of both sides of \eqref{field1} we can simply write the equation of motion as:
\begin{equation}
\label{field2}
\nabla\times[\frac{1}{\R}\Delta\bfv+\epsilon(\Delta{\bf
w}\times\bfv+\Delta\bfv_t)-\bfv_t-{\bf w}\times\bfv]=0,
\end{equation}
which is the field equation of incompressible unsteady second grade Rivlin-Ericksen fluid independent of the choice
of any particular coordinate system.

Now we restrict our interest to flows in a cylindrical tube and assume that the velocity is dependent only on two cross-sectional variables $x$, $y$ and the time $t$. The incompressiblity of the fluid allows us to introduce a stream function $\psi$ such that $\bfv=(\psi_y,-\psi_x,w)$ where $\psi=\psi(t,x,y)$ and also $w=w(t,x,y)$.

We further take the cross section of the cylinder to be a disk with unit radius and consider the no-slip boundary conditions. Then the equations \eqref{field2} admit the following steady state solution
\begin{equation*}
  w_0=\frac{1}{2}(1-x^2-y^2),\qquad \psi_0=0.
\end{equation*}
Here the characteristic velocity has been chosen as $U=\frac{p L^2}{4 \mu}$. First considering the deviation $w'=w-w_0$ and $\psi'=\psi-\psi_0$ and then introducing the polar coordinates $w''(t, r, \theta) = w'(t, r\cos\theta, r\sin\theta)$ and $\psi''(t, r,\theta) = \psi'(t, r\cos\theta, r \sin \theta)$ and ignoring the primes, the equations become
\begin{equation} \label{main}
  \begin{aligned}
    & \frac{\partial}{\partial t}(1-\epsilon \Delta) w = \frac{1}{\R} \Delta w + \R \psi_{\theta} + J (\psi, (1-\epsilon \Delta)w), \\
    & \frac{\partial}{\partial t}\Delta (\epsilon \Delta - 1) \psi = -\frac{1}{\R} \Delta^2 \psi + \epsilon \R \Delta w_{\theta} + J ((1-\epsilon \Delta) \Delta \psi, \psi) + \epsilon J(\Delta w, w),
  \end{aligned}
\end{equation}
in the interior of the unit disk $\Omega$ where $J$ is the advection operator
\[
  J(f, g) = \frac{1}{r} (f_r g_{\theta} - f_{\theta} g_r).
\]
The field equations are supplemented with no-slip boundary conditions for the velocity field
\begin{equation} \label{BC}
  w = \psi = \frac{\partial \psi}{\partial r} = 0 \text{ at } r = 1.
\end{equation}

In this paper, our main aim is to investigate the stability and transitions of \eqref{main} subject to \eqref{BC}.
We first prove that the system undergoes a dynamic transition at the critical Reynolds number $\R_c\approx4.124 \epsilon^{-1/4}$.
As $\R$ crosses $\R_c$ the steady flow loses its stability, and a transition occurs.
If we denote the azimuthal wavenumber of an eigenmode by $m$, then two modes, called critical modes hereafter, with $m=3$ and radial wavenumber $1$, become critical at $\R=\R_c$.
Using the language of dynamic transition theory \cite{ptd}, we can show that the transition is either Type-I(continuous) or Type-II(catastrophic).
In Type-I transitions, the amplitudes of the transition states stay close to the base flow after the transition. Type-II transitions, on the other hand, are associated with more complex dynamical behavior, leading to metastable states or a local attractor far away from the base flow.

We show that the type of transition preferred in system \eqref{main} is determined by the real part of a complex parameter $A$ which only depends on $\epsilon$.
In the generic case of nonzero imaginary part of $A$, there are two possible transition scenarios depending on the sign of the real part of $A$: continuous or catastrophic.
In the continuous transition scenario, a stable, small amplitude, time periodic flow with 3-fold azimuthal symmetry bifurcates on $\R>\R_c$.
The time period of the bifurcated solution tends to infinity as $\R$ approaches $\R_c$, a phenomenon known as infinite period bifurcation \cite{keener1981infinite}.
The dual scenario is the catastrophic transition where the bifurcation is subcritical on $\R<\R_c$ and a repeller bifurcates.
In the non-generic case where the imaginary part of $A$ vanishes, the limit cycle degenerates to a circle of steady states.

The transition number $A$ depends on the system parameter $\epsilon$ in a non-trivial way, hence it is not possible to find an analytical expression of $A$ as a function of $\epsilon$.
So, $A$ must be computed numerically for a given $\epsilon$.
Physically, the transition number can be considered as a measure of net mechanical energy transferred from all modes back to the critical modes which in turn modify the base flow.
We show that $A$ is determined by the nonlinear interactions of the critical modes ($m=3$) with all the modes having $m=0$ and $m=6$.
Moreover, our numerical computations suggest that for low $\epsilon$ fluids ($\epsilon <1$), just a single nonlinear interaction, namely the one with $m=0$ and radial wavenumber $1$ mode, dominates all the rest contributions to $A$.
Our numerical experiments with low $\epsilon$, i.e. $\epsilon < 1$, suggest that the real part of $A$ is
always positive indicating a catastrophic transition at $\R=\R_c$.

We also determine the Reynolds number threshold $\R_E>0$ below which the Poiseuille flow is globally stable, attracting all initial conditions with at least exponential convergence in the $H_0^1(\Omega)$ norm for the velocity.
We find that $R_E \approx 12.87$ when $\epsilon=0$. The gap between $\R_E$ and $\R_c$ shrinks to zero quickly as $\epsilon$ is increased.

The paper is organized as follows: In Section~3, the linearized stability of the system is studied and the principle of exchange of stabilities is investigated. In Section~4, transition theorem of the system is presented with its proof given in Section~5. Section~6 is devoted to the energy stability of the system. In Section~7, we give a detailed numerical analysis. Finally in Section~8 the conclusions and possible extensions of this study are discussed in detail.

\section{The model and the functional setting}
Throughout $\Re z$, $\Im z$, $\overline{z}$ will denote the real part, imaginary part and the complex conjugate of a complex number $z$.

Let
\[
  X_1 = H_0^1(\Omega) \times H_0^2(\Omega), \qquad
  X = L^2(\Omega)\times L^2(\Omega),
\]
where $\Omega$ is the unit disk in $\mathbb{R}^2$, $H_0^1(\Omega)$ and $H_0^2(\Omega)$ denote the usual Sobolev spaces and $L^2(\Omega)$ is the space of Lebesgue integrable functions.

For $\phi_i = \left[ \begin{smallmatrix} w_i(r, \theta) \\ \psi_i(r, \theta) \end{smallmatrix} \right] \in X$, $i=1,2$, the inner product on $X$ is defined by
\begin{equation} \label{inner product}
  \langle \phi_1, \phi_2 \rangle
  =\int_0^{2\pi}\int_0^1 (w_1 \overline{w_2} + \psi_1 \overline{\psi_2})r dr d\theta,
\end{equation}
with the norm on $X$ defined as $\norm{\phi}^2 = \langle \phi, \phi \rangle$.
We define the linear operators $M: X_1 \to X$ and $N: X_1 \to X$ as
\begin{equation} \label{operators}
  M =
  \begin{bmatrix}
    I-\epsilon \Delta & 0 \\
    0 & \Delta (\epsilon \Delta - I)
  \end{bmatrix},
  \quad
  N =
  \begin{bmatrix}
    \frac{1}{\R}\Delta & \R \partial_{\theta} \\
    \epsilon \R \Delta \partial_{\theta} & -\frac{1}{\R}\Delta^2
  \end{bmatrix},
\end{equation}
and the nonlinear operator $H: X_1 \to X$ as
\begin{equation*}
  H(\phi) =
  \begin{bmatrix}
    J(\psi, (1-\epsilon \Delta)w) \\
    J((1-\epsilon\Delta)\Delta w, \psi) + \epsilon J(\Delta w, w)
  \end{bmatrix},
  \qquad \text{for }
      \phi = \begin{bmatrix} w \\ \psi \end{bmatrix} \in X_1.
\end{equation*}
We will use $H$ to denote both the nonlinear operator as well as the bilinear form
\begin{equation} \label{bilinear form}
  H(\phi_I, \phi_J) =
  \begin{bmatrix}
    J(\psi_I, (1-\epsilon \Delta)w_J) \\
    J((1-\epsilon\Delta)\Delta w_I, \psi_J) + \epsilon J(\Delta w_I, w_J)
  \end{bmatrix},
\end{equation}
for $\phi_I = [w_I, \psi_I]^T,$ $\phi_J = [w_J, \psi_J]^T$.
Also we will use the symmetrization of this bilinear form
\begin{equation} \label{Hs}
  H_s(\phi_I, \phi_J) = H(\phi_I, \phi_J) + H(\phi_J, \phi_I).
\end{equation}
Then the equations \eqref{main} and \eqref{BC} can be written in the following abstract form
\begin{equation} \label{abstract equ1}
  M \phi_t = N\phi + H(\phi), \qquad \phi \in X_1,
\end{equation}
with initial condition
\[
  \phi(0) = \phi_0 \in X_1.
\]

\section{Linear Stability}
To determine the transitions of \eqref{abstract equ1}, the first step is to study the eigenvalue problem $N \phi = \beta M \phi$ of the linearized operator.
This is equivalent to the problem
\begin{equation} \label{eigenvalue problem}
  \begin{aligned}
    & \frac{1}{\R} \Delta w + \R \psi_{\theta} = \beta (1 -  \epsilon \Delta) w, \\
    & \frac{1}{\R} \Delta^{2} \psi - \epsilon \R\Delta w_{\theta} = \beta(1- \epsilon \Delta)\Delta\psi,
  \end{aligned}
\end{equation}
subject to the boundary conditions \eqref{BC}.

An interesting feature of the eigenvalue problem is the following.
\begin{lemma} \label{Thm: real eigenvalues}
  Any eigenvalue $\beta$ of \eqref{eigenvalue problem} with boundary conditions \eqref{BC} is real.
\end{lemma}
\begin{proof}
  Multiplying the first equation of \eqref{eigenvalue problem} by $\overline{\Delta w}$ and the second equation by $\overline{\psi}$, integrating over the domain $\Omega$, we obtain after integration by parts
  \begin{equation} \label{proof: eq1}
    (\frac{1}{\R} + \beta \epsilon) \norm{\Delta w}^2 + \R \int_\Omega \psi_{\theta} \overline{\Delta w} dr d\theta = - \beta \norm{\nabla w}^2,
  \end{equation}
  and
  \begin{equation} \label{proof: eq2}
    -(\frac{1}{\R} + \beta \epsilon) \norm{\Delta \psi}^2 + \epsilon \R \int_{\Omega} \Delta w_{\theta} \overline{\psi} dr d\theta = \beta \norm{\nabla \psi}^2.
  \end{equation}
  Let
  \[
    \begin{aligned}
      & \mathcal{A}_1 = \epsilon \norm{\Delta w}^2 + \norm{\Delta \psi}^2, \\
      & \mathcal{A}_2 = \epsilon \norm{\nabla w}^2 + \norm{\nabla \psi}^2 \\
      & \mathcal{A}_3 = 2\epsilon \int_\Omega \Delta w_{\theta} \overline{\psi} dr d\theta.
    \end{aligned}
  \]
  Now consider $-\epsilon \times$\eqref{proof: eq1} + \eqref{proof: eq2} which is
  \begin{equation} \label{proof: eq3}
    -(\frac{1}{\R} + \beta \epsilon) \mathcal{A}_1 + \R Re(\mathcal{A}_3) = \beta \mathcal{A}_2,
  \end{equation}
  after integrating by parts.
  Taking the imaginary part of \eqref{proof: eq3} gives
  \[
    \Im(\beta) (\epsilon \mathcal{A}_1 + \mathcal{A}_2) = 0.
  \]
  Since $(\epsilon \mathcal{A}_1 + \mathcal{A}_2) \ge 0$, we must have $\Im(\beta) = 0$.
\end{proof}

Now we turn to the problem of determining explicit expressions of the solutions of the eigenvalue problem of the linearized operator.  Thanks to the periodicity in the $\theta$ variable, for $m \in \mathbb{Z}$ and $j \in \mathbb{Z}_+$, we denote the eigenvectors of \eqref{eigenvalue problem} by
\begin{equation} \label{ansatz}
  \phi_{m,j}(r, \theta) = e^{i m \theta} \varphi_{m,j}(r), \qquad
  \varphi_{m,j}(r)=
  \begin{bmatrix}
    w_{m,j}(r) \\
    \psi_{m,j}(r)
  \end{bmatrix}
\end{equation}
with corresponding eigenvalues $\beta_{m,j}$. Let us set the eigenvalues $\beta_{m,j}$ for $m\neq0$ to be ordered so that
\[
  \beta_{m,1}  \ge \beta_{m2} \ge \cdots
\]
for each $m \in \mathbb{Z}$.

Plugging the ansatz \eqref{ansatz} into \eqref{eigenvalue problem} and omitting $j$ we obtain two ODE's in the $r$-variable.
\begin{equation} \label{seperated EVP}
  \begin{aligned}
    & \left(\frac{1}{\R} + \beta \epsilon \right) \Delta_m w_m + i m \R \psi_m = \beta w_m, \\
    -& \left( \frac{1}{\R} + \beta \epsilon \right) \Delta_m^2 \psi_m + i \epsilon m \R \Delta_m w_m = -\beta \Delta_m \psi_m,
  \end{aligned}
\end{equation}
with boundary conditions
\begin{equation} \label{seperated EVP BC}
  w_m(1) = \psi_m(1) = \psi_m'(1) = 0
\end{equation}
where
\begin{equation*}
\begin{aligned}
  & \Delta_m = \frac{d^2}{dr^2} + \frac{1}{r} \frac{d}{dr} - \frac{m^2}{r^2}.
\end{aligned}
\end{equation*}

When $m=0$, the equations \eqref{seperated EVP} become decoupled and we easily find that there are two sets of eigenpairs given by
\begin{equation*}
\begin{aligned}
  & w^1_{0,j}(r) = J_0(\alpha_{0,j} r), \qquad \psi^1_{0,j}(r) = 0, \qquad \beta^1_{0,j} = \frac{-\alpha_{0,j}^2}{\R(1+\epsilon \alpha_{0,j}^2)}, \\
  & w^2_{0,j}(r) = 0, \qquad \psi^2_{0,j}(r) = J_0(\alpha_{1,j} r) - J_0(\alpha_{1,j}), \qquad \beta^2_{0,j} = \frac{-\alpha_{1,j}^2}{\R(1+\epsilon \alpha_{1,j}^2)},
\end{aligned}
\end{equation*}
where $\alpha_{k,j}$ is the $j$th zero of the $k$th Bessel function $J_k$. In particular, $\beta_{0,j} < 0$ for all $\epsilon$ and for all $\R$.

In the $m \neq 0$ case, as the eigenvalues are real by Lemma~\ref{Thm: real eigenvalues}, the multiplicity of each eigenvalue $\beta=\beta_{m,j}=\beta_{-m,j} \in \mathbb{R}$ is generically two with corresponding eigenvectors $\phi_{m,j}$ and $\phi_{-m,j} = \overline{\phi_{m,j}}$.

Solving the first equation of \eqref{seperated EVP} for $\psi_m$ and plugging it into the second equation of \eqref{seperated EVP} yields a sixth order equation
\begin{equation} \label{compact equation for w_m}
  ( \lambda  + \Delta_m ) ( \mu + \Delta_m) \Delta_m w_m = 0
\end{equation}
where
\begin{equation} \label{lambda mu}
  \lambda = \frac{\sqrt{\epsilon} m \R - \beta_{m,j}}{\frac{1}{\R} + \epsilon \beta_{m,j}}, \qquad
  \mu = \frac{-\sqrt{\epsilon} m \R - \beta_{m,j}}{\frac{1}{\R} + \epsilon \beta_{m,j}}.
\end{equation}
It is easy to check that $\lambda=0$ or $\mu=0$ yield only trivial solutions to equations \eqref{seperated EVP BC} and \eqref{compact equation for w_m}. So we will assume $\lambda \neq 0$, and $\mu \neq 0$.

When $m>0$, the general solution of \eqref{compact equation for w_m} is
\[
  w_m = c_1 r^m + c_2 J_m(\sqrt{\lambda} r) + c_3 J_m(\sqrt{\mu} r) + c_4 r^{-m} + c_5 Y_m(\sqrt{\lambda}r) + c_6 Y_m(\sqrt{\mu}r),
\]
where $J_m$ and $Y_m$ are the Bessel functions of the first and the second kind respectively.

The boundedness of the solution and its derivatives at $r=0$ implies that $c_4 = c_5 = c_6 = 0$ and     we get the eigensolutions
\begin{equation} \label{eigensolutions mneq0}
  \begin{aligned}
    & w_{m,j}(r) =
    \begin{cases}
      c_1 r^m + c_2 J_m(\sqrt{\lambda} r) + c_3 J_m(\sqrt{\mu} r) & \text{if } m>0\\
      \overline{w_{-m,j}}, & \text{if } m<0
    \end{cases} \\
    & \psi_{m,j}(r) =
    \begin{cases}
      d_1 r^m + d_2 J_m(\sqrt{\lambda} r) + d_3 J_m(\sqrt{\mu} r) & \text{if } m>0\\
      \overline{\psi_{-m,j}}, & \text{if } m<0
    \end{cases}
  \end{aligned}
\end{equation}
with $d_1 = \dfrac{-i \beta_{m,j} c_1}{m \R}$, $d_2 = -i \sqrt{\epsilon} c_2$, $d_3 = i \sqrt{\epsilon} c_3$.


The eigenvalues and two of the three coefficients $c_1$, $c_2$, $c_3$ in \eqref{eigensolutions mneq0} are determined by the boundary conditions \eqref{seperated EVP BC} which form a linear system for the coefficients $c_k$. This system has a nontrivial solution only when the dispersion relation
\begin{equation} \label{dispersion relation 1}
  \begin{vmatrix}
  1 & J_m(\sqrt{\lambda}) & J_m(\sqrt{\mu}) \\
  \beta & \sqrt{\epsilon} m \R J_m(\sqrt{\lambda} ) & -\sqrt{\epsilon} m \R J_m(\sqrt{\mu} ) \\
  \beta m & \sqrt{\epsilon} m \R \sqrt{\lambda} J'_m(\sqrt{\lambda} ) & -\sqrt{\epsilon} m \R \sqrt{\mu}J'_m(\sqrt{\mu} )
  \end{vmatrix}
  = 0,
\end{equation}
is satisfied.

Using the identity
\[
  J_m'(z) = \frac{m}{z}J_m(z) - J_{m+1}(z)
\]
we can show that \eqref{dispersion relation 1} is equivalent to
\begin{equation} \label{dispersion relation}
  \sqrt{\lambda} J_m(\sqrt{\lambda}) J_{m+1}(\sqrt{\mu}) + \sqrt{\mu} J_m(\sqrt{\mu})J_{m+1}(\sqrt{\lambda}) = 0,
\end{equation}
where $J_m$ is the Bessel function of the first kind of order $m$.

To compute the critical Reynolds number $\R_c$, we set $\beta=0$ in \eqref{dispersion relation} which, after some manipulation, yields
\begin{equation} \label{dispersion relation beta=0 case}
  I_m(\sqrt{\lambda}) J'_m(\sqrt{\lambda}) - J_m(\sqrt{\lambda})I'_m(\sqrt{\lambda}) = 0.
\end{equation}
Once \eqref{dispersion relation beta=0 case} is solved for $\lambda$, the corresponding Reynolds number is obtained by the relation $\lambda = \sqrt{\epsilon}m \R^2$ (from \eqref{lambda mu} when $\beta=0$).
We note here that this is the exact same equation as the one obtained in \cite{ozer1999}.
For each $m$, the equation \eqref{dispersion relation beta=0 case} has infinitely many solutions $\{\lambda_{m, j}\}_{j=1}^{\infty}$ where $\lambda_{m,j}$ increases with $j$ and $\lambda_{m,j} \to \infty$ as $j\to\infty$. Letting $\R_{m,j} = \epsilon^{-1/4} (\lambda_{m,j}/m)^{1/2}$,
 we have $\beta_{m,j}=0$ when $\R=\R_{m,j}$.
\begin{table}[t]
  \begin{tabular}{|l||l|l|l|l|l|l|}
    \hline
    $m$ & 1 & 2 & 3 & 4 & 5 & 6 \\
    \hline
    $\lambda_{m,1}$ & 21.260 & 34.877 & 51.030 & 69.665 & 90.739 & 114.21 \\
    \hline
    $(\lambda_{m,1}/m)^{1/2}$ & 4.610 & 4.175 & 4.124 & 4.173 & 4.260 & 4.36\\
    \hline
  \end{tabular}
  \caption{The smallest positive root $\lambda_{m,1}$ of \eqref{dispersion relation beta=0 case} and $(\lambda_{m,1}/m)^{1/2}$ for $m=1,\dots,5$.}
  \label{Table:lambdamj}
\end{table}

We define the critical Reynolds number
\[
  \R_c = \min_{m,j} \R_{m,j} = \min_m \R_{m,1} = \epsilon^{-1/4}\min_{m \in \mathbb{Z}_{+}} \sqrt{\frac{\lambda_{m,1}}{m}},
\]
so that $\beta_{m,j}\le0$ if $\R \le \R_c$.
There has been recent progress on the properties of zeros of \eqref{dispersion relation beta=0 case}.
In \cite{baricz2015cross}, the estimate
\begin{equation} \label{lambdam1 bounds}
\begin{split}
  & \frac{2^4 (m+1)(m+2)(m+3)\sqrt{(m+4)(m+5)}}{\sqrt{5m+15}}< \lambda_{m,1}^2 \\
  & \qquad \qquad < \frac{2^4(m+1)(m+2)(m+3)(m+4)(m+5)}{5m+17},
\end{split}
\end{equation}
on $\lambda_{m,1}$ is obtained.

Using \eqref{lambdam1 bounds}, we can show that the upper bound for $\sqrt{\lambda_{m,1}/m}$ for $m=3$ is less than its
lower bound for all $m\ge7$ which implies that $\sqrt{\lambda_{3,1}/3} < \sqrt{\lambda_{m,1}/m}$ for all $m\ge7$. Thus $\R_c$ is minimized for some $m$ smaller than $7$ and hence can be found by brute force.
Looking at Table~\ref{Table:lambdamj}, we find that the expression above is indeed minimized when $m=3$ and obtain the relation
\eqref{Rc}, i.e. $\R_c = \R_{3,1} \approx 4.124 \epsilon^{-1/4}$.

Defining the left hand side of \eqref{dispersion relation} as $\omega(\beta, \R)$, the equation \eqref{dispersion relation} becomes $\omega(\beta,\R) = 0$. By the implicit function theorem, this defines $\beta_{3,1}(\R)$ for $\R$ near $\R_c$ with $\beta_{3,1}(\R_c)=0$. With the aid of symbolic computation, we can compute
\[
  \frac{d\beta_{3,1}}{d \R} \mid_{\R=\R_c} =
  \dfrac{\partial \omega/\partial R}{\partial \omega/ \partial \beta} \mid_{\R=\R_c, \beta=\beta_{3,1}} =
  \frac{0.12 \sqrt{\epsilon}}{0.02+\epsilon} >0.
\]
Thus we have proved the Principal of Exchange of Stabilities which we state below.
\begin{theorem} \label{Thm: PES}
For $\epsilon \neq 0$, let
\begin{equation} \label{Rc}
   \R_c = \epsilon^{-1/4} \sqrt{\frac{\lambda_{3,1}}{3}} \approx 4.124 \epsilon^{-1/4}.
\end{equation}
Then
  \begin{equation} \label{PES}
    \begin{aligned}
      & \beta_{3,1}(\R) = \beta_{-3,1}(\R)
      \begin{cases}
        < 0 & \text{if } \R < \R_c \\
        = 0 & \text{if } \R = \R_c \\
        > 0 & \text{if } \R > \R_c
      \end{cases} \\
      & \beta_{m,j}(\R_c) < 0 , \qquad \text{if } (m,j) \neq (\pm 3,1).
    \end{aligned}
  \end{equation}
\end{theorem}

Note that Theorem~\ref{Thm: PES} is in contrast to the $\epsilon=0$ case where the basic flow is linearly stable for all Reynolds numbers. This can be seen easily by noting that when $\epsilon=0$, the inner product of the second equation in \eqref{eigenvalue problem} with $\psi$ yields
  \begin{equation} \label{eps=0 estimate}
    \norm{\Delta \psi}^2 = - \beta \R \norm{\nabla \psi}^2.
  \end{equation}
With the Dirichlet boundary conditions, from \eqref{eps=0 estimate}, it follows easily that $\beta < 0$ if $\psi \neq 0$.

For the proof and presentation of Theorem~\ref{thm: main}, we also need to solve the eigenvalue problem of the adjoint linear operator which yields adjoint modes orthogonal to the eigenmodes of the linear operator.
Adjoint problem is obtained by taking the inner product of \eqref{abstract equ1} by $\phi^{\ast}$ and moving the derivatives via integration by parts onto $\phi^{\ast}$ by making use of the boundary conditions.
This yields the following adjoint problem
\begin{equation*}
  N^{\ast} \phi^{\ast} = \beta^{\ast} M^{\ast} \phi^{\ast},
\end{equation*}
where
\begin{equation*}
  M^{\ast} = M =
  \begin{bmatrix}
    I-\epsilon \Delta & 0 \\
    0 & \Delta (\epsilon \Delta - I)
  \end{bmatrix},
  \qquad
  N^{\ast} =
  \begin{bmatrix}
    \frac{1}{\R}\Delta & -\epsilon \R \Delta \partial_{\theta} \\
    \R \partial_{\theta} & -\frac{1}{\R}\Delta^2
  \end{bmatrix},
  \qquad
  \phi^{\ast} =
  \begin{bmatrix}
    w^{\ast} \\ \psi^{\ast}
  \end{bmatrix},
\end{equation*}
and $w^{\ast}$, $\phi^{\ast}$ satisfies the same boundary conditions \eqref{BC} as $w$ and $\phi$.
We denote the adjoint eigenvectors by $\phi^{\ast}_{m,j} = \begin{bmatrix} w_{m,j}^{\ast} & \psi_{m,j}^{\ast} \end{bmatrix}^T$ and we also have the adjoint eigenvalues $\beta_{m,j}^{\ast} = \beta_{m,j}$. The reason we introduce the adjoint eigenmodes is to make use of the following orthogonality relation
\begin{equation} \label{orthogonality}
  \ip{\phi_{m,i}, M \phi_{n,j}^{\ast}} = 0 \quad \text{if } (m,i) \neq (n,j).
\end{equation}

\section{Dynamic Transitions}
Let us briefly recall here the classification of dynamic transitions and refer to \cite{ptd} for a detailed rigorous discussion. For $\epsilon \neq 0$, as the critical Reynolds number $\R_c$ is crossed, the principle of exchange of stabilities \eqref{PES} dictates that the nonlinear system always undergoes a dynamic transition, leading to one of the three type of transitions, Type-I(continuous), II(catastrophic) or III(random). On $\R>\R_c$, the transition states stay close to the base state for a Type-I transition and leave a local neighborhood of the base state for a Type-II transition. For Type-III transitions, a local neighborhood of the base state is divided into two open regions with a Type-I transition in one region, and a Type-II transition in the other region. Type-II and Type-III transitions are associated with more complex dynamics behavior.

Below we prove that for \eqref{main}, only two scenarios are possible. In the first scenario, the system exhibits a Type-I (or continuous) transition and a stable attractor will bifurcate on $\R>\R_c$ which attracts all sufficiently small disturbances to the Poiseuille flow. We prove that this attractor is homeomorphic to the circle $S^1$ and is generically a periodic orbit. The Figure~\ref{fig:bifsol} shows the stream function of the bifurcated time-periodic solution. The dual scenario is that the system exhibits a Type-II (or catastrophic) transition.

The type of transition at $\R=\R_c$ depends on the transition number
\begin{equation} \label{A}
  A = \sum_{j=1}^{\infty} A_{0,j} + A_{6,j},
\end{equation}
where $A_{m,j}$ represent the nonlinear interaction of the critical modes with the mode with azimuthal wavenumber $m$ and radial wavenumber $j$. The formulas for $A_{m,j}$ are
\begin{equation} \label{A0 A6}
  \begin{aligned}
    & A_{0,j} = \frac{1}{\langle \phi_{3,1}, M \phi_{3,1}^{\ast} \rangle} \tilde{\Phi}_{0,j} \langle H_s(\phi_{3,1}, \phi_{0,j}), \phi_{3,1}^{\ast} \rangle \\
    & A_{6,j} = \frac{1}{\langle \phi_{3,1}, M \phi_{3,1}^{\ast} \rangle} \tilde{\Phi}_{6,j} \langle H_s(\overline{\phi_{3,1}}, \phi_{6,j}), \phi_{3,1}^{\ast} \rangle \\
    & \tilde{\Phi}_{0,j} = \frac{1}{-\beta_{0,j} \langle \phi_{0,j}, M\phi_{0,j}^{\ast} \rangle} \langle H_s(\phi_{3,1}, \overline{\phi_{3,1}}), \phi_{0,j}^{\ast} \rangle, \\
    & \tilde{\Phi}_{6,j} = \frac{1}{-\beta_{6,j} \langle \phi_{6,j}, M\phi_{6,j}^{\ast} \rangle} \langle H(\phi_{3,1}, \phi_{3,1}), \phi_{6,j}^{\ast} \rangle, \\
  \end{aligned}
\end{equation}

To recall the meaning of various terms in \eqref{A0 A6}, $\beta_{n,j}$ is the $j$th eigenvalue of the $n$th azimuthal mode with corresponding eigenvector $\phi_{n,j}$ and adjoint eigenvector $\phi_{n,j}^{\ast}$. $\langle \cdot, \cdot \rangle$ denotes the inner product \eqref{inner product}, $H$ denotes the bilinear form \eqref{bilinear form}, $H_s$ is its symmetrization \eqref{Hs}, $M$ is the linear operator defined by \eqref{operators}.

\begin{theorem} \label{thm: main}
  If $\epsilon \neq 0$ then the following statements hold true.
  \begin{enumerate}
    \item If $\Re(A) < 0$ then the transition at $\R=\R_c$ is Type-I and an attractor $\Sigma_R$ bifurcates on $\R>\R_c$ which is homeomorphic to $S^1$. If $\Im(A)=0$ then  $\Sigma_R$ is a cycle of steady states. If $\Im(A) \neq 0$ then $\Sigma_R$ is the orbit of a stable limit cycle given by
    \begin{equation} \label{bifurcated solution}
      \begin{bmatrix}
        w_{\text{per}}(t,r,\theta) \\
        \psi_{\text{per}}(t,r,\theta)
      \end{bmatrix}
      = 2 \sqrt{\frac{-\beta_{3,1}}{\Re(A)}}\Re\left(\exp\left(\frac{2 i \pi t}{T}  \right) e^{i 3 \theta}
      \begin{bmatrix}
        w_{3,1}(r) \\ \psi_{3,1}(r)
       \end{bmatrix} \right)
      + o(\sqrt{-\beta_{3,1}}),
    \end{equation}
    with time period
    \begin{equation} \label{time period}
      T = \dfrac{2 \pi \Re(A)}{\Im(A) \beta_{3,1}}.
    \end{equation}
    Here $w_{3,1}$ and $\psi_{3,1}$ are the vertical velocity and stream function of the eigenmode of the linear operator with corresponding eigenvalue $\beta_{3,1}$.

    \item If $\Re(A) > 0$ then the transition at $\R=\R_c$ is Type-II and a repeller $\Sigma_R$ bifurcates on $\R<\R_c$. If $\Im(A) = 0$ then $\Sigma_R$ is a cycle of steady states. If $\Im(A) \neq 0$ then $\Sigma_R$ is the orbit of of an unstable limit cycle  given by \eqref{bifurcated solution} with $\beta_{3,1}$ replaced by $-\beta_{3,1}$.
  \end{enumerate}
\end{theorem}

\begin{figure}[t]
  \centering
    \includegraphics[scale=1]{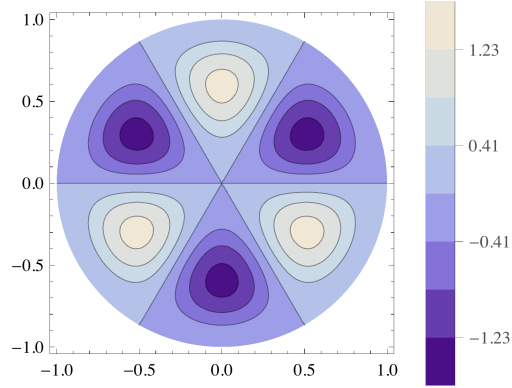}
  \caption{The time periodic stream function $\psi_{per}$ given by \eqref{bifurcated solution} which rotates in time, clockwise if $T>0$ and counterclockwise if $T<0$.}
  \label{fig:bifsol}
\end{figure}
\begin{remark}
  In the generic case of $\Im(A) \neq 0$, Theorem~\ref{thm: main} guarantees the existence of a stable (unstable)
  bifurcated periodic solution on $\R>\R_c$ ($\R<\R_c$). By \eqref{time period}, the period of the bifurcated solution approaches to infinity as $\R \downarrow \R_c$ ($\R \uparrow \R_c$).
\end{remark}
\section{Proof Of Theorem~\ref{thm: main}}
As is standard in the dynamic transition approach, the proof of Theorem~\ref{thm: main} depends on the reduction of the field equations \eqref{abstract equ1} on to the center manifold.

Let us denote the (real) eigenfunctions and adjoint eigenfunctions corresponding to the critical eigenvalue $\beta_{3,1}$ by
\begin{align*}
  & e_1(r, \theta) = \Re(\phi_{3,1}(r, \theta)) \qquad
  e_2(r, \theta) = \Im (\phi_{3,1}(r, \theta)) \\
  & e_1^{\ast}(r, \theta) = \Re(\phi^{\ast}_{3,1}(r, \theta)) \qquad
  e_2^{\ast}(r, \theta) = \Im (\phi^{\ast}_{3,1}(r, \theta))
\end{align*}
By the spectral theorem, the spaces $X_1$ and $X$ can be decomposed into the direct sum
\[
  X_1 = E_1 \oplus E_2, \qquad
  X = E_1 \oplus \overline{E_2},
\]
where
\begin{align*}
  & E_1 = \text{span} \{e_1, e_2\}, \\
  & E_2 = \{u \in X_1  \mid \langle u, e_i^{\ast} \rangle = 0 \quad i=1,2\}, \\
  & \overline{E_2} = \text{closure of $E_2$ in $X$}.
\end{align*}
Since $M:X_1 \to X$ is an invertible operator, we can define $L=M^{-1}N$ and $G=M^{-1}H$. Now the abstract equation \eqref{abstract equ1} can be written as
\begin{equation} \label{abstract equ 2}
  \frac{d\phi}{d t} = L \phi + G(\phi).
\end{equation}
The linear operator $L$ in \eqref{abstract equ 2} can be decomposed into
\begin{align*}
  & L = J \oplus \mathcal{L}, \\
  & J = L \mid_{E_1}: E_1 \to E_1, \\
  & \mathcal{L} = L \mid_{E_2} : E_2 \to \overline{E_2},
\end{align*}
Since the eigenvalues are real, we have $L e_k = \beta_{3,1} e_k$ for $k=1,2$. Hence we have $J = \beta_{3,1}(\R) I_2$, where $I_2$ is the $2\times 2$ identity matrix. We know that when $J$ is diagonal, we have the following approximation of the center manifold function $\Phi: E_1 \to E_2$ near $\R \approx \R_c$; see \cite{ptd}.
\begin{equation} \label{CM approximation formula}
  -\mathcal{L} \Phi(x) = P_2 G_k(x) + o(k),
\end{equation}
The meaning of the terms in the above formula \eqref{CM approximation formula} are as follows.
\begin{itemize}
  \item[a)] $o(k) = o(\norm{x}^k) + O(\abs{\beta_{3, 1}(\R)} \norm{x}^k) \quad \text{as } \R \to \R_c, \, x\to 0,$
  \item[b)] $P_2 : X \to E_2$ is the canonical projection,
  \item[c)] $x$ is the projection of the solution onto $E_1$,
  \begin{equation} \label{center part}
    x(t, r, \theta) = x_1(t) e_1(r, \theta) + x_2(t) e_2(r, \theta)
  \end{equation}
  \item[d)] $G_k$ denotes the lowest term of the Taylor expansion of $G(u)$ around $u=0$. In our case $G$ is bilinear and thus $k=2$ in  and $G=G_k$.
\end{itemize}

It is easier to carry out the reduction using complex variables. So we write \eqref{center part} as
\begin{equation} \label{center part 2}
  x(t, r, \theta) =
  z(t) \phi_{3,1}(r, \theta) + \overline{z(t) \phi_{3,1}(r, \theta)}
\end{equation}
where $z(t) = \frac{1}{2}(x_1(t) - i x_2(t))$. Let us expand the center manifold function by
\begin{equation} \label{CM expansion}
  \Phi = \sum_{(n,j) \neq (\pm 3,1)} \Phi_{n,j}(t) \phi_{n,j}(r, \theta)
\end{equation}
Plugging the above expansion into the center manifold approximation formula \eqref{CM approximation formula},  taking inner product with $M \phi_{n,j}^{\ast}$
and using the orthogonality \eqref{orthogonality} we have
\begin{equation} \label{CM coefficients formula}
  \Phi_{n,j} = \frac{1}{-\beta_{n,j} \ip{\phi_{n,j}, M \phi_{n,j}^{\ast}} } \ip{H(x), \phi_{n,j}^{\ast}} + o(2).
\end{equation}
Since $H$ is bilinear,
\begin{equation} \label{Hx}
  H(x) = H(z \phi_{3,1} + \overline{z \phi_{3,1}}) = z^2 H(\phi_{3,1}, \phi_{3,1}) + z\overline{z} H_s(\phi_{3,1}, \overline{\phi_{3,1}}) + \overline{z}^2 H(\overline{\phi_{3,1}}, \overline{\phi_{3,1}}),
\end{equation}
with the operator $H_s$ defined by \eqref{Hs}.
Thanks to the orthogonality $\int_0^{2\pi}e^{i n \theta} e^{-im \theta} d\theta = 2\pi \delta_{nm}$, we have
\begin{equation} \label{H-orthogonality}
  \langle H(\phi_{m_1, i_1}, \phi_{m_2, i_2}), \phi_{m_3, i_3}^{\ast} \rangle = 0, \qquad \text{if } m_1+m_2 \neq m_3
\end{equation}
With $\overline{\phi_{3,1}}=\phi_{-3,1}$, this implies
\begin{equation} \label{secondary modes}
  \ip{H(x), \phi_{n,j}^{\ast}} = 0 \qquad \text{if } n \notin \{0, -6, 6\}.
\end{equation}
According to \eqref{secondary modes}, \eqref{CM expansion} and \eqref{CM coefficients formula},
\begin{equation} \label{CM expansion 2}
  \Phi(t) = \sum_{j=1}^{\infty} \Phi_{0,j}(t) \phi_{0,j} +
  \Phi_{6,j}(t) \phi_{6,j} + \Phi_{-6,j}(t) \phi_{-6,j} + o(2)
\end{equation}
That is the center manifold is $o(2)$ in eigendirections whose azimuthal wavenumber is not $0$, $6$ or $-6$.
The equation \eqref{Hx} implies that
\begin{equation} \label{cm-comp 1}
  \begin{aligned}
    & \langle H(x), \phi_{0,j}^{\ast} \rangle =
    z \overline{z} \langle H_s(\phi_{3,1}, \overline{\phi_{3,1}}) , \phi_{0,j}^{\ast} \rangle, \\
    & \langle H(x), \phi_{6,j}^{\ast} \rangle =
    z^2 \langle H(\phi_{3,1}, \phi_{3,1}), \phi_{6,j}^{\ast} \rangle, \\
    & \langle H(x), \phi_{- 6,j}^{\ast} \rangle =
    \overline{\langle H(x), \phi_{6,j}^{\ast} \rangle}.
  \end{aligned}
\end{equation}
By \eqref{cm-comp 1} and \eqref{CM coefficients formula}, we get the coefficients of the center manifold in \eqref{CM expansion 2}
\begin{equation} \label{cm-variables}
  \begin{aligned}
    & \Phi_{0,j} = z \overline{z} \tilde{\Phi}_{0,j} + o(2), \\
    & \Phi_{6,j} = z^2 \tilde{\Phi}_{6,j} + o(2), \\
    & \Phi_{-6,j} = \overline{\Phi_{6,j}},
  \end{aligned}
\end{equation}
where $\tilde{\Phi}_{0,j}$ and $\tilde{\Phi}_{6,j}$ are given by \eqref{A0 A6}.

As the dynamics of the system is enslaved to the center manifold for small initial data and for Reynolds numbers close to the critical Reynolds number $\R_c$, it is sufficient to investigate the dynamics of the main equation \eqref{abstract equ1} on the center manifold. For this reason we take
\[
  \phi(t) = x(t) + \Phi(t),
\]
in \eqref{abstract equ1} to obtain
\begin{equation} \label{pre-reduced}
  \frac{dz}{dt} M \phi_{3,1} + \frac{d\overline{z}}{dt} M \overline{\phi_{3,1}} = z N \phi_{3,1} + \overline{z} N \overline{\phi_{3,1}} + H(x+ \Phi).
\end{equation}
To project the above equation onto the center-unstable space $E_1$, we take inner product of \eqref{pre-reduced} with $\phi_{3,1}^{\ast}$ and use
\[
  \langle M \overline{\phi_{3,1}}, \phi_{3,1}^{\ast} \rangle = 0,
\]
and
\[
  N\phi_{3,1} = \beta_{3,1} M \phi_{3,1}, \qquad
  N\overline{\phi_{3,1}} = \beta_{3,1} M \overline{\phi_{3,1}},
\]
to get the following reduced equation of \eqref{abstract equ1}.
\begin{equation} \label{reduced-1}
  \frac{dz}{dt} = \beta_{3,1}(\R) z + \frac{1}{\langle \phi_{3,1}, M \phi_{3,1}^{\ast} \rangle} \langle H(x + \Phi), \phi_{3,1}^{\ast}\rangle.
\end{equation}
The reduced equation \eqref{reduced-1} describes the transitions of the full nonlinear system for $\R$ near $\R_c$ and small initial data. At this stage, the nonlinear term in \eqref{reduced-1} is too complicated to explicitly describe the transition. Thus we need to determine the lowest order expansion in $z$ of the nonlinear term $\langle H(x + \Phi), \phi_{3,1}^{\ast}\rangle$. By the bilinearity of $H$,
\begin{equation} \label{nonlinear comp1}
  \langle H(x + \Phi), \phi_{3,1}^{\ast}\rangle =
  \langle H(x), \phi_{3,1}^{\ast}\rangle + \langle H_s(x, \Phi), \phi_{3,1}^{\ast}\rangle +\langle H(\Phi), \phi_{3,1}^{\ast}\rangle.
\end{equation}
The first term in \eqref{nonlinear comp1} vanish by \eqref{secondary modes} and the last term in \eqref{nonlinear comp1} is $o(3)$ as $H(\Phi) = o(3)$ since $\Phi=O(2)$ and $H$ is bilinear. Thus \eqref{nonlinear comp1} becomes
\begin{equation} \label{nonlinear comp2}
  \langle H(x + \Phi), \phi_{3,1}^{\ast}\rangle =
  \langle H_s(x, \Phi), \phi_{3,1}^{\ast} \rangle + o(3).
\end{equation}
Using the expression \eqref{center part 2} for $x$, we can rewrite \eqref{nonlinear comp2} as
\begin{equation} \label{comp1}
\begin{aligned}
  \langle H(x + \Phi), \phi_{3,1}^{\ast}\rangle = z \langle H_s(\phi_{3,1}, \Phi), \phi_{3,1}^{\ast} \rangle
  + \overline{z}
  \langle H_s(\overline{\phi_{3,1}}, \Phi), \phi_{3,1}^{\ast} \rangle + o(3).
\end{aligned}
\end{equation}
Now we use the expansion \eqref{CM expansion 2} of $\Phi$ in \eqref{comp1} and the orthogonality relations
\[
  \langle H_s(\phi_{3,1}, \phi_{n,j}), \phi_{3,1}^{\ast} \rangle = 0 \qquad \text{if } n \neq 0,
\]
and
\[
  \langle H_s(\overline{\phi_{3,1}}, \phi_{n,j}), \phi_{3,1}^{\ast} \rangle = 0 \qquad \text{if } n \neq 6,
\]
which follow from \eqref{H-orthogonality} to arrive at
\begin{equation} \label{nonlinear comp3}
  \langle H(x + \Phi), \phi_{3,1}^{\ast}\rangle = \sum_{j=1}^{\infty} z \Phi_{0,j} \langle H_s(\phi_{3,1}, \phi_{0,j}), \phi_{3,1}^{\ast} \rangle + \overline{z} \Phi_{6,j} \langle H_s(\overline{\phi_{3,1}}, \phi_{6,j}), \phi_{3,1}^{\ast} \rangle + o(3).
\end{equation}
Defining the coefficient $A$ by \eqref{A} and making use of \eqref{cm-variables} and \eqref{nonlinear comp3}, we write down the approximate equation of \eqref{reduced-1} as
\begin{equation} \label{reduced equation}
  \frac{dz}{dt} = \beta_{3,1}(\R) z + A \abs{z}^2 z + o(3).
\end{equation}

To finalize the proof, there remains to analyze the stability of the zero solution of \eqref{reduced equation} for small initial data. In polar coordinates $z(t) = \abs{z}e^{i \gamma}$, \eqref{reduced equation} is equivalent to
\begin{equation} \label{reduced-final}
  \frac{d\abs{z}}{dt} = \beta_{3,1}(\R) \abs{z} + \Re(A) \abs{z}^3 + o(\abs{z}^3), \qquad
  \frac{d\gamma}{dt} = \Im(A) \abs{z}^2 + o(\abs{z}^3).
\end{equation}
For $\R>\R_c$ as $\beta_{3,1}>0$, it is clear from \eqref{reduced-final} that $z=0$ is unstable if $\Re(A)>0$ and is locally stable if $\Re(A)<0$. In the latter case, the bifurcated solution is
\[
  z(t) = \sqrt{\frac{-\beta_{3,1}(\R)}{\Re(A)}} \exp \left(-i \frac{\Im(A)}{\Re(A)}\beta_{3,1}(\R) t \right).
\]

Thus to determine the stability of the bifurcated state as $\R$ crosses the critical Reynolds number $\R_c$, we need to compute the sign of the real part of $A$.

The details of the assertions in the proof of Theorem~\ref{thm: main} follow from the attractor bifurcation theorem in \cite{ptd}. That finishes the proof.

\section{Energy Stability}
In this section we study the energy stability of the equations \eqref{main} which is related to at least exponential decay of solutions to the base flow. We refer to \cite{straughan2013energy} for a multitude of applications of this theory.

For $f$, $g$, $h$ in $H_0^1(\Omega)$, the following two properties of $J$ follows from integrating by parts
\begin{equation} \label{orthogonality property 1 of J}
  \begin{aligned}
    & \langle J(f, g), h \rangle = - \langle J(f, h), g \rangle, \\
    & \langle J(f, g), h \rangle = 0, \quad \text{if $f$, $g$ and $h$ are linearly dependent}.
  \end{aligned}
\end{equation}

Taking inner product of the first equation in \eqref{main} with $w$ and the second equation in \eqref{main} with $\psi$ and using the property \eqref{orthogonality property 1 of J}, we have
\begin{equation} \label{energy 0-1}
  \frac{1}{2}\frac{d}{dt} (\norm{w}^2 + \epsilon \norm{\nabla w}^2) = -\frac{1}{\R} \norm{\nabla w}^2 + \R \ip{\psi_{\theta},w} + \epsilon\ip{J(\Delta w, \psi),w},
\end{equation}
\begin{equation} \label{energy 0-2}
  \frac{1}{2}\frac{d}{dt} (\norm{\nabla \psi}^2 + \epsilon \norm{\Delta \psi}^2) = -\frac{1}{\R} \norm{\Delta \psi}^2 + \epsilon \R \ip{\Delta w_{\theta},\psi} + \epsilon \ip{J(\Delta w, w),\psi}.
\end{equation}

Adding equations \eqref{energy 0-1} and \eqref{energy 0-2} and using \eqref{orthogonality property 1 of J} once again, we arrive at
\begin{equation} \label{energy 0-3}
  \frac{1}{2}\frac{d}{dt} \mathcal{E}(t) =  -\frac{1}{\R} \mathcal{I}_1(t) + \R \mathcal{I}_2(t),
\end{equation}
where
\begin{equation*}
  \begin{aligned}
    & \mathcal{E} = \norm{w}^2 + \epsilon \norm{\nabla w}^2 + \norm{\nabla \psi}^2 + \epsilon \norm{\Delta \psi}^2 \\
    & \mathcal{I}_1 = \norm{\nabla w}^2 + \norm{\Delta \psi}^2 \\
    & \mathcal{I}_2 = \ip{\psi_{\theta},w-\epsilon \Delta w}.
  \end{aligned}
\end{equation*}

Letting
\begin{equation} \label{RE}
  \frac{1}{\R_E^2} = \max_{X_1 \setminus\{0\}} \frac{\mathcal{I}_2}{\mathcal{I}_1},
\end{equation}
we have by \eqref{energy 0-3}
\begin{equation} \label{energy-1}
  \frac{d}{dt} \mathcal{E} \le -2\R(\frac{1}{\R^2} - \frac{1}{\R_E^2}) \mathcal{I}_1.
\end{equation}
Since $\mathcal{I}_1 \ge 0$ and $\mathcal{I}_2 = 0$ whenever $\psi_{\theta}=0$, $\R_E$ must be nonnegative.

Since $w \in H_0^1(\Omega)$ and $\nabla \psi \in H_0^1(\Omega)$, by the Poincaré inequality, $\abs{\nabla w}^2 \ge \eta_1 \abs{w}^2$ and $\abs{\Delta \psi}^2 \ge \eta_1 \abs{\nabla \psi}^2$, where $\eta_1 \approx 5.78$ is the first eigenvalue of negative Laplacian on $\Omega$. Thus we have
\begin{equation} \label{energy-2}
  \mathcal{I}_1 \ge \frac{\eta_1}{1+ \epsilon \eta_1}\mathcal{E}
\end{equation}
Now let
\[
  c_R = \frac{2 \R \eta_1}{1+ \epsilon \eta_1}  (\frac{1}{\R^2} - \frac{1}{\R_E^2}),
\]
and suppose that $\R<\R_E$. Then $c_R>0$ and by \eqref{energy-1} and \eqref{energy-2},
\[
  \frac{d}{dt} \mathcal{E}(t) \le - c_R \mathcal{E}(t).
\]
Hence the Gronwall's inequality implies
\[
  \mathcal{E}(t) \le e^{-c_R t}\mathcal{E}(0).
\]
In particular, for $\R \le \R_E$, $c_R>0$ and any initial disturbance in $X_1 $ will decay to zero implying the unconditional stability of the basic steady state solution.

Using the variational methods to maximize the quantity in \eqref{RE}, we find the resulting Euler-Lagrange equations as
\begin{equation} \label{EL-0}
  \begin{aligned}
    & \Delta w + \frac{\R^2}{2} (1- \epsilon \Delta) \psi_{\theta} = 0, \\
    & \Delta^2 \psi + \frac{\R^2}{2}(1 - \epsilon \Delta) w_{\theta} = 0.
  \end{aligned}
\end{equation}
Considering \eqref{EL-0} as an eigenvalue problem with $\R$ playing the role of the eigenvalue, $\R_E$ is just the smallest positive eigenvalue. To solve \eqref{EL-0}, we plug the ansatz $w = e^{i m \theta} w_m(r)$ and $\psi = e^{i m \theta} \psi_m(r)$ into \eqref{EL-0} which yields
\begin{equation} \label{EL-1}
  \begin{aligned}
    & \Delta_m w_m + i \frac{m \R^2}{2}(1-\epsilon \Delta_m) \psi_m = 0 \\
    & \Delta_m^2 \psi_m + i \frac{m \R^2}{2}(1-\epsilon \Delta_m) w_m = 0,
  \end{aligned}
\end{equation}
where $\Delta_m = \frac{d^2}{dr^2} + \frac{1}{r} \frac{d}{dr} - \frac{m^2}{r^2}$.
Taking $\Delta_m$ of the second equation above and using the first equation, we obtain
\begin{equation} \label{EL-2}
  p(\Delta_m) \psi_m = 0
\end{equation}
where
\begin{equation*}
  p(\xi) = \xi^3 + \frac{m^2 \R^4}{4}  (1-\epsilon \xi)^2.
\end{equation*}
Let $\xi_1$, $\xi_2$ and $\xi_3$ be the three roots of $p$. As the discriminant of $p$ is negative, one root is real and the others are complex conjugate. The factorization of the operator in \eqref{EL-2} gives
\begin{equation} \label{EL-3}
  (\Delta_m - \xi_1)(\Delta_m - \xi_2)(\Delta_m - \xi_3) \psi_m = 0.
\end{equation}

The general solution of \eqref{EL-3} is
\begin{equation*}
  \psi_m = \sum_{k=1}^3 c_k I_m(\sqrt{\xi_k} r) + \tilde{c}_k K_m(\sqrt{\xi_k} r),
\end{equation*}
where $I_m$ and $K_m$ are the modified Bessel functions.
The boundedness of the solution at $r=0$ necessitates $\tilde{c}_k = 0$ for $k=1,2,3$. Thus
\begin{equation*}
  \psi_m = \sum_{k=1}^3 c_k I_m(\sqrt{\xi_k} r).
\end{equation*}
Now applying the operator $1-\epsilon \Delta_m$ to the second equation in \eqref{EL-1} and using the first equation of \eqref{EL-1}, we obtain $p(\Delta_m) w_m = 0$, i.e. the same equation \eqref{EL-2}, this time for $w_m$. Hence
\begin{equation*}
  w_m = \sum_{k=1}^3 d_k I_m(\sqrt{\xi_k} r).
\end{equation*}
The first equation in \eqref{EL-1} gives the relation $d_k = -i \frac{m \R^2}{2} (\epsilon-\xi_k^{-1}) c_k$ between $c_k$'s and $d_k$'s.
Now the boundary conditions $w_m(1)=\psi_m(1)=\psi_m'(1)=0$ constitute a homogeneous system of three linear equations for the coefficients $c_k$'s.
The existence of nontrivial solutions is then equivalent to the vanishing determinant of this system which after some manipulation becomes
\begin{equation} \label{EL-determinant}
  \begin{vmatrix}
    \xi_1^{-1} I_m(\sqrt{\xi_1}) & \xi_2^{-1} I_m(\sqrt{\xi_2}) & \xi_3^{-1} I_m(\sqrt{\xi_3}) \\
    I_m(\sqrt{\xi_1}) & I_m(\sqrt{\xi_2}) & I_m(\sqrt{\xi_3}) \\
    \sqrt{\xi_1}I_m'(\sqrt{\xi_1}) & \sqrt{\xi_2}I_m'(\sqrt{\xi_2}) & \sqrt{\xi_3}I_m'(\sqrt{\xi_3}) \\
  \end{vmatrix} = 0.
\end{equation}

For fixed $m$ and $\epsilon$, the equation \eqref{EL-determinant} has infinitely many solutions $\R=\R_{m,j}(\epsilon)$, $j \in \mathbb{Z}_+$. Letting
\begin{equation} \label{Rm}
  \R_m = \min_{j \in \mathbb{Z}_+} \R_{m,j},
\end{equation}
the critical Reynolds number is given by
\begin{equation} \label{RE num}
  \R_E = \min_{m \in \mathbb{Z}_+} \R_m.
\end{equation}
We present the numerical computations of $\R_E$ in the next section.

\section{Numerical Results and Discussion}

\subsection{Transitions at the critical Reynolds number $\R_c$.}
As shown by Theorem~\ref{thm: main}, the transition at the critical Reynolds number $\R_c \approx 4.124 \epsilon^{-1/4}$ is dictated by the real part of $A$ given by \eqref{A}.

Let us define
\begin{equation*}
  A = \lim_{N \to \infty} A^N,
\end{equation*}
where $A^N$ is the series in \eqref{A} truncated at $N$, i.e. $A^N= \sum_{j=1}^N A_{0,j} + A_{6,j}$ and $A_{m,j}$ represents the nonlinear interaction of the critical modes and the mode with azimuthal wavenumber $m$ and radial wavenumber $j$ given by \eqref{A0 A6}.

A symbolic computation software is used to compute $A^N$. We present our numerical computations of $A^N$ in Figure~\ref{fig:A vs eps} for $\epsilon=$ $1,\, 10^{-1},\, 10^{-2},\, 10^{-3}$ and $1\le N\le10$.
The imaginary part of $A$ is nonzero and we are only interested in the sign of the real part of $A$ to determine the type of transition according to Theorem~\ref{thm: main}.
To simplify the presentation, we scale all $A^N$'s so that $\abs{\Re(A^1)} = 1$. The plots in Figure~\ref{fig:A vs eps} suggest that the convergence of the truncations $A^N \to A$ is rapid for small $\epsilon$ but a higher order truncation (larger $N$) is necessary to accurately resolve $A$ for larger $\epsilon$. For $\epsilon < 10^{-1}$, even $A^1$ is a good approximation to determine the sign of $A$.
For example, the relative error for approximating $A$ with $A^1$ is approximately \%2 for $\epsilon=10^{-3}$ and increases to approximately \%18 for $\epsilon=1$.

We also measure the relative strength of the nonlinear interactions, i.e. the ratio
\begin{equation} \label{BN}
  B^N = \frac{ \sum_{j=1}^N \Re(A_{6,j})}{\sum_{j=1}^N \Re(A_{0,j})} ,
\end{equation}
in Figure~\ref{fig:BN}.

It is seen from Figure~\ref{fig:BN} that the contribution from the modes with $m=0$ dominates when $\epsilon$ is low. But as $\epsilon$ increases, the contribution from modes with $m=6$ start to become significant. For example, for $\epsilon=10^{-3}$, $B^N$ approaches $8\times 10^{-5}$, for $\epsilon=10^{-2}$, $B^N$ approaches $-2\times 10^{-3}$ and for $\epsilon=10^{-1}$, $B^N$ approaches $-8\times10^{-3}$. In particular, for low $\epsilon$, we have $A \approx A^1 \approx A_{0,1}$.

More significantly, our numerical results presented in Figure~\ref{fig:A vs eps} show that the real part of $A$ is positive for $\epsilon =10^{-3}, 10^{-2}, 10^{-1}, 1$, meaning that the transition is catastrophic by Theorem~\ref{thm: main}.
Thus the system moves to a flow regime away from the base Poiseuille flow and the system exhibits complex dynamical behavior for $\R>\R_c$.
\begin{figure*}[t]
    \centering
    \begin{tikzpicture}
    \begin{axis}[
        xlabel=$N$,
        ylabel =$\Re(A^N)$,
        legend style={at={(.9,0.5)}},
        ymin=0.8,
        ymax=1.05]
    \addplot coordinates {(1, 1.) (2, 1.00089) (3, 1.01997) (4, 1.02112) (5, 1.02068) (6, 1.02065) (7, 1.02065) (8, 1.02064) (9, 1.02064) (10, 1.02064)};

    \addplot coordinates {(1, 1.) (2, 0.997596) (3, 1.00991) (4, 1.01071) (5, 1.00798) (6, 1.00777) (7, 1.00771) (8, 1.00769) (9, 1.00767) (10, 1.00767)};

    \addplot[mark=o] coordinates {(1, 1.) (2, 0.989211) (3,
    0.975863) (4, 0.974511) (5,   0.959659) (6, 0.958498) (7, 0.958186) (8, 0.958068) (9, 0.958014) (10, 0.957986)};

    \addplot coordinates {(1, 1.) (2, 0.961376) (3, 0.872072) (4, 0.865378) (5, 0.824498) (6, 0.821406) (7, 0.820612) (8, 0.820488) (9, 0.820427) (10, 0.820395)};
    \legend{$\epsilon=0.001$,$\epsilon=0.01$,$\epsilon=0.1$, $\epsilon=1$}
    \end{axis}
    \end{tikzpicture}
    \caption{The plots of $\Re(A^N)$ for different $\epsilon$ values. All $A^N$'s are scaled so that $\abs{\Re(A^1)} = 1$.}
    \label{fig:A vs eps}
\end{figure*}
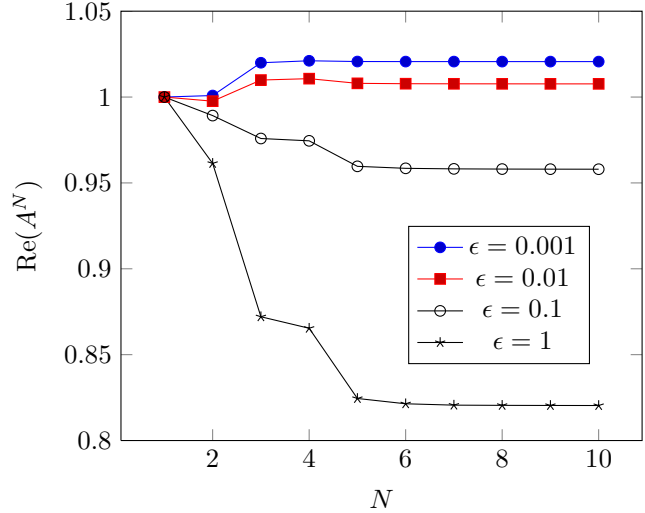

\begin{figure*}[t]
    \centering
    \begin{tikzpicture}
    \begin{axis}[
        xlabel=$N$,
        ylabel =$B^N$,
        legend style={at={(.9,0.5)}}
        ]
    \addplot coordinates {(1, 0.0000751446) (2, 0.0000574423) (3, 0.0000597475) (4,  0.0000867288) (5, 0.000087839) (6, 0.0000885365) (7, 0.0000872772) (8, 0.0000870878) (9, 0.0000866053) (10, 0.0000865978)};

    \addplot coordinates {(1, -0.00217747) (2, -0.00231445) (3, -0.00232501) (4, -0.0023344) (5, -0.00234389) (6, -0.00235146) (7, -0.00235559) (8, -0.00235673) (9, -0.00235925) (10, -0.00236105)};

    \addplot coordinates {(1, -0.00716897) (2, -0.00768623) (3, -0.00790949) (4, -0.00793913) (5, -0.00807035) (6, -0.0080853) (7, -0.00809099) (8, -0.00809388) (9, -0.00809573) (10, -0.00809694)};
    \legend{$\epsilon=0.001$,$\epsilon=0.01$,$\epsilon=0.1$}
    \end{axis}
    \end{tikzpicture}
    \caption{$B^N$, defined by \eqref{BN}, measures the relative strength of the nonlinear interactions of the critical modes with $m=6$ modes to the $m=0$ modes.}
    \label{fig:BN}
\end{figure*}
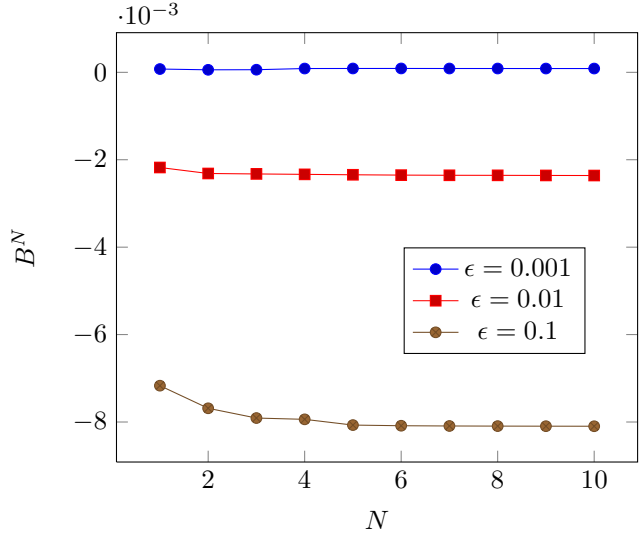

\subsection{Determination of Energy Stability Threshold $\R_E$.}
With a standard numerics package, $\R_m(\epsilon)$ in \eqref{Rm} can be computed for given $m$ and $\epsilon$.
Then by \eqref{RE num}, $\R_E$ is computed by taking minimum in \eqref{RE num} over all (computed) $\R_m$. In Table~\ref{Table:Rm RE Rc}, it is shown that for $\epsilon=10^{-4}$, and $\epsilon=10^{-3}$, $\R_E$ is obtained for $m=1$ while for $\epsilon = 10^{-2}$ and $\epsilon=2\times 10^{-2}$, $\R_E$ is obtained for $m=2$.
In Figure~\ref{fig:Rm vs eps}, we plot $\R_m$ ($m=1,2,3$) for $0 < \epsilon \le 5\times 10^{-2}$.
We see that the curves $\R_1$ and $\R_2$ intersect approximately at $\epsilon=0.009$ while $\R_2$ and $\R_3$ intersect approximately at $\epsilon=0.024$.
As $\epsilon$ is increased, the value of $m$ for which $\R_E$ is minimized also increases.
For higher values of $m$, the roots of the determinant in \eqref{EL-determinant} becomes increasingly hard to find.

In Table~\ref{Table:Rm RE Rc}, the last column gives the value of $\R_c$, the linear instability threshold, computed by \eqref{Rc}.
Note that the interval $[\R_E, \R_c]$ consists of Reynolds numbers for which the base flow is either not globally stable or globally stable but not not exponentially attracting.
We plot the $R_E$ and $R_c$ data from Table~\ref{Table:Rm RE Rc} in Figure~\ref{fig:RE vs Rc} which shows that this interval shrinks rapidly as $\epsilon$ is increased.

\begin{table}[ht]
  \begin{tabular}{|c||c|c|c|c|c||c||c|}
    \hline
    $\epsilon$ & $\R_1$ & $\R_2$ & $\R_3$ & $\R_4$ & $\R_5$ & $\R_E$ & $\R_c$\\
    \hline
    \hline
    $0$ & 12.87 & 13.49 & 14.84 & 16.37 & 17.95 & 12.87 & $\infty$ \\
    \hline
    $10^{-4}$ & 12.86 & 13.47 & 14.81 & 16.32 & 17.88 & 12.86 & 42.4 \\
    \hline
    $10^{-3}$ & 12.77 & 13.31 & 14.54 & 15.90 & 17.29 & 12.77 & 23.84\\
    \hline
    $10^{-2}$ & 11.99 & 11.95 & 12.44 & 12.95 & 13.39 & 11.95 & 13.40 \\
    \hline
    $2\times10^{-2}$ & 11.26 & 10.83 & 10.91 & 11.03 & 11.13 & 10.83 & 11.27 \\
    \hline
  \end{tabular}
  \caption{$\R_m$ denotes the first positive root of \eqref{EL-determinant} and $\R_E$ is the minimum of $\R_m$ taken over all $m$. $\R_c$ is the linear instability threshold.}
  \label{Table:Rm RE Rc}
\end{table}

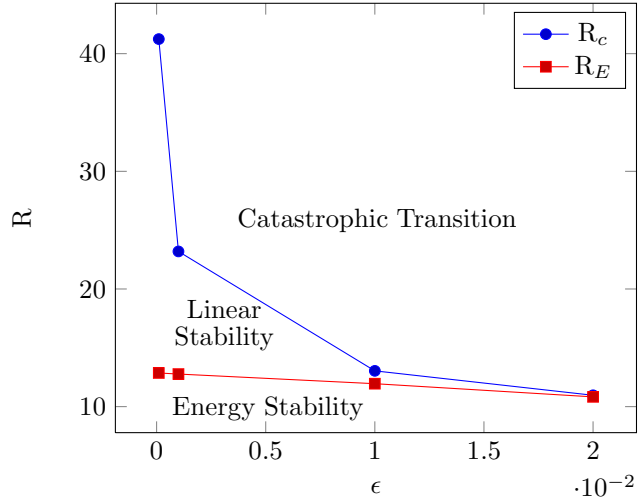
\begin{figure}[ht]
  \centering
  \begin{tikzpicture}
  \begin{axis}[
      xlabel=$\epsilon$,
      ylabel =R]
  \addplot coordinates {(0.0001, 41.24) (0.001, 23.191) (0.01, 13.0412) (0.02, 10.9663)};
  \addplot coordinates {(0.0001, 12.86) (0.001, 12.77) (0.01, 11.95) (0.02, 10.83)};
  \node at (50, -10) {Energy Stability};
  \node at (30, 75) {Linear};
  \node at (30, 50) {Stability};
  \node at (100, 150) {Catastrophic Transition};
  \legend{R$_c$,R$_E$}
  \end{axis}
  \end{tikzpicture}
  \caption{$\R_E$ and $R_c$ curves in the $\epsilon-\R$ plane.}
  \label{fig:RE vs Rc}
\end{figure}

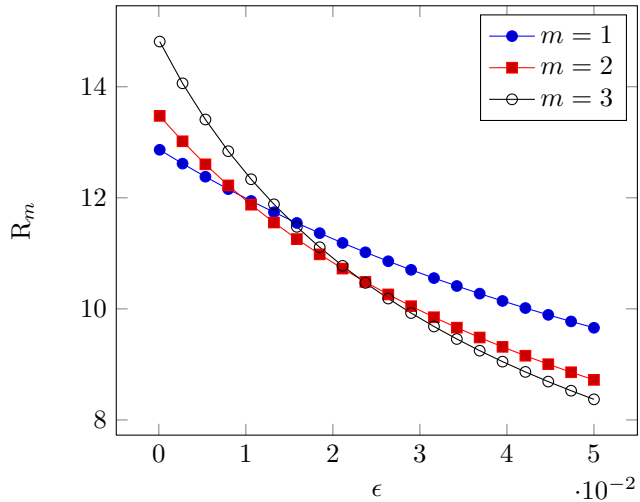
\begin{figure}[ht]
  \centering
  \begin{tikzpicture}
  \begin{axis}[
      xlabel=$\epsilon$,
      ylabel =R$_m$]
  \addplot coordinates {(0.0001, 12.8648) (0.00272632, 12.6159) (0.00535263,   12.3801) (0.00797895, 12.1562) (0.0106053, 11.9435) (0.0132316,   11.741) (0.0158579, 11.548) (0.0184842, 11.3638) (0.0211105, 11.1877) (0.0237368, 11.0193) (0.0263632, 10.8579) (0.0289895, 10.7032) (0.0316158, 10.5547) (0.0342421, 10.412) (0.0368684, 10.2747) (0.0394947, 10.1426) (0.0421211, 10.0153) (0.0447474, 9.89248) (0.0473737, 9.77399) (0.05, 9.65954)};

  \addplot coordinates {(0.0001, 13.4751) (0.00272632, 13.0184) (0.00535263,   12.6032) (0.00797895, 12.224) (0.0106053, 11.8759) (0.0132316,   11.5551) (0.0158579, 11.2582) (0.0184842, 10.9824) (0.0211105,   10.7255) (0.0237368, 10.4855) (0.0263632, 10.2605) (0.0289895,   10.0491) (0.0316158, 9.84999) (0.0342421, 9.66208) (0.0368684,   9.48434) (0.0394947, 9.31592) (0.0421211, 9.15603) (0.0447474, 9.00398) (0.0473737, 8.85916) (0.05, 8.72102)};

  \addplot[mark=o] coordinates {(0.0001, 14.8142) (0.00272632, 14.062) (0.00535263,   13.4105) (0.00797895, 12.8397) (0.0106053, 12.3345) (0.0132316, 11.8834) (0.0158579, 11.4775) (0.0184842, 11.1099) (0.0211105, 10.775) (0.0237368, 10.4682) (0.0263632, 10.1859) (0.0289895, 9.92498) (0.0316158, 9.68294) (0.0342421, 9.45761) (0.0368684, 9.24717) (0.0394947, 9.05006) (0.0421211, 8.86493) (0.0447474, 8.69064) (0.0473737, 8.52616) (0.05, 8.37062)};
  \legend{$m=1$,$m=2$,$m=3$}
  \end{axis}
  \end{tikzpicture}
  \caption{The plot of $\R_m$ vs $\epsilon$ for $m=1,2,3$.}
  \label{fig:Rm vs eps}
\end{figure}

\section{Concluding Remarks}

In this work, we considered both the energy stability and transitions of the Poiseuille flow of a second grade fluid
in an infinite circular pipe with the restriction that the flow is independent of the axial variable $z$.

We show that unlike the Newtonian ($\epsilon=0$) case, in the second grade model ($\epsilon \neq 0$ case),
the time independent base flow exhibits transitions as the Reynolds number $\R$ exceeds the critical threshold $\R_c \approx 4.124 \epsilon^{-1/4}$ where $\epsilon$ is a material constant measuring the relative strength of second order viscous effects compared to inertial effects.
At $\R=\R_c$ we prove that a transition takes place and that the type of the transition depends on a complex number $A$.
In particular depending on $A$, generically, either a continuous transition to a periodic solution or a catastrophic transition occurs where the bifurcation is subcritical.
The time period of the periodic solution approaches to infinity as $\R$ approaches $\R_c$, a phenomenon known
 as infinite period bifurcation.

We show that the number $A = \sum_{j=1}^{\infty} A_{0,j} + A_{6,j}$ where $A_{m,j}$ denotes the nonlinear interaction of the two critical modes with the mode having azimuthal wavenumber $m$ and radial wavenumber $j$.
Our numerical results suggest that for low $\epsilon$ ($\epsilon<1$), $A$ is approximated well by $A_{0,1}$.
That is, a single nonlinear interaction between the critical modes and the mode with azimuthal wavenumber $0$ and radial wavenumber $1$ dominates all the rest of interactions and hence determines the transition.

Our numerical results suggest that a catastrophic transition is preferred for low $\epsilon$ values ($\epsilon < 1$).
This means, an unstable time periodic solution bifurcates on $\R<\R_c$.
On $\R>\R_c$, the system has either metastable states or a local attractor far away from the base flow and a more complex dynamics emerges.

We also show that for $\R < \R_E$ with $\R_E < \R_c$, the Poiseuille flow is at least exponentially globally stable in the $H_0^1(\Omega)$ norm for the velocity.
We find that $\R_E \approx 12.87$ when $\epsilon=0$ and the gap between $\R_E$ and $\R_c$ diminishes quickly as $\epsilon$ is increased.

There are several directions in which this work can further be extended.

\textit{First}, in this work we consider a pipe with a circular cross section of unit radius and find that the first two critical modes have azimuthal wavenumber equal to $3$.
Increasing (or decreasing) the radius of the cross section will also increase (decrease) the azimuthal wavenumber of the critical modes.
The analysis in this case would be similar to our presentation except in the case where four critical modes with azimuthal wave numbers $m$ and $m+1$ become critical. In the case of four critical modes, more complex patterns will emerge due to the cross interaction of the critical modes \cite{sengul-rb3d, sengul-marangoni}.
An analysis in the light of \cite{sengul-rb2d} is required to determine transitions in this case of higher multiplicity criticality.

\textit{Second}, for the Reynolds number region between $\R_E$ and $\R_c$, there may be regions where the base flow is either globally stable but not exponentially attractive or regions where the domain of attraction of the base flow is not the whole space.
A conditional energy stability analysis is required to resolve these Reynolds number regimes \cite{straughan2013energy}.

\textit{Third}, the results we proved in this work for the second grade fluids can also be extended to fluids of higher grades and to other types of shear flows.

\textit{Fourth}, in this work we restricted attention to 2D flows.
In the expense of complicating computations and results, a similar analysis could be considered for 3D flows which depend also on the axial variable $z$.

\bibliography{poiseuille.bib}{}
\bibliographystyle{siam}

\end{document}